\documentclass[12pt]{article}
\usepackage[centertags]{amsmath}
\usepackage{amsfonts}
\usepackage{amssymb}
\usepackage{amsthm}
\usepackage{newlfont}
\usepackage{graphicx}

\newfont{\bb}{msbm10 at 12pt}
\def\r{\hbox{\bb R}}

\def\e{\hbox{\bf E}}

\def\v{\hbox{\bf V}}

\newtheorem{theorem}{Theorem}[section]
\newtheorem{definition}[theorem]{Definition}

\newtheorem{corollary}[theorem]{Corollary}

\begin{document}

\title{Some Characterizations of Cylindrical Helices  in $\e^n$}
\author{ Ahmad T. Ali\\Mathematics Department\\
 Faculty of Science, Al-Azhar University\\
 Nasr City, 11448, Cairo, Egypt\\
email: atali71@yahoo.com\\
\vspace*{1cm}\\
 Rafael L\'opez\footnote{Partially
supported by MEC-FEDER
 grant no. MTM2007-61775 and
Junta de Andaluc\'{\i}a grant no. P06-FQM-01642.}\\
Departamento de Geometr\'{\i}a y Topolog\'{\i}a\\
Universidad de Granada\\
18071 Granada, Spain\\
email: rcamino@ugr.es}

\maketitle
\begin{abstract} We consider a unit speed curve $\alpha$ in Euclidean $n$-dimensional space $\e^n$ and denote the Frenet frame by $\{\v_1,\ldots,\v_n\}$. We say that $\alpha$ is a cylindrical helix if its tangent vector $\v_1$ makes a constant angle with a fixed direction $U$. In this work we give different characterizations of such curves in terms of their curvatures.

\end{abstract}



\section{Introduction and statement of results}
An helix in Euclidean 3-space $\e^3$ is a curve where the tangent lines make a constant angle with a fixed direction. An helix curve is characterized by the fact that the ratio $\kappa/\tau$ is constant along the curve, where $\kappa$ and $\tau$ are the curvature and the torsion of
$\alpha$, respectively. Helices are well known curves in classical differential geometry of space curves \cite{mp} and we refer to the reader for recent works on this type of curves \cite{gl1,sc}. Recently, Magden \cite{ma} have introduced the concept of cylindrical helix in Euclidean 4-space $\e^4$ saying that the tangent lines makes a constant angle with a fixed directions. He characterizes a cylindrical helix in $\e^4$ if and only if the function
\begin{equation}\label{eqma}
\Big(\dfrac{\kappa_1}{\kappa_2}\Big)^2+\Big(\dfrac{1}{\kappa_3}\Big(\dfrac{\kappa_1}{\kappa_2}\Big)^{\prime}\Big)^2
\end{equation}
is constant along the curve , where $\kappa_3$ and $\kappa_4$ are the third and the fourth curvature of the the curve. See also \cite{oh}.

In this work we consider the generalization of the concept of general helices in Euclidean n-space $\e^n$. Let $\alpha:I\subset\r\rightarrow\e^n$ be an arbitrary curve in $\e^n$. Recall that the curve $\alpha$ is said to be of unit speed (or parameterized by arc-length function $s$) if $\langle\alpha'(s),\alpha'(s)\rangle=1$, where $\langle,\rangle$ is the standard scalar product in  Euclidean space $\e^n$ given by
$$
\langle X,Y\rangle=\sum_{i=1}^{n}\,x_i\,y_i,
$$
for each $X=(x_1,\ldots,x_n)$, $Y=(y_1,\ldots,y_n)\in\e^n$.

Let $\{\v_1(s),\ldots,\v_n(s)\}$ be the moving frame along $\alpha$, where the vectors $\v_i$ are mutually orthogonal vectors satisfying $\langle\v_i,\v_i\rangle=1$.
The Frenet equations for $\alpha$ are given by (\cite{gl1})
\begin{equation}\label{u2}
 \left[
   \begin{array}{c}
     \v_1' \\
     \v_2' \\
     \v_3' \\
     \vdots\\
     \v_{n-1}'\\
     \v_{n}'\\
   \end{array}
 \right]=\left[
           \begin{array}{cccccccc}
             0 & \kappa_1 & 0 & 0 & \cdots & 0 & 0\\
             -\kappa_1 & 0 & \kappa_2 & 0 & \cdots & 0 & 0\\
             0 & -\kappa_2 & 0 & \kappa_3 & \cdots & 0 & 0\\
             \vdots & \vdots & \vdots & \vdots & \ddots & \vdots & \vdots\\
             0 & 0 & 0 & 0 & \cdots & 0 & \kappa_{n-1}\\
             0 & 0 & 0 & 0 & \cdots & -\kappa_{n-1} & 0\\
           \end{array}
         \right]\left[
   \begin{array}{c}
     \v_1 \\
     \v_2 \\
     \v_3 \\
     \vdots\\
     \v_{n-1}\\
     \v_{n}\\
   \end{array}
 \right].
 \end{equation}

Recall the functions $\kappa_i(s)$ are called the i-th curvatures of $\alpha$. If $\kappa_{n-1}(s)=0$ for any $s\in I$, then $\v_n(s)$ is a constant vector $V$ and the curve $\alpha$ lies in a $(n-1)$-dimensional affine subspace orthogonal to $V$, which is isometric to the Euclidean $(n-1)$-space $\e^{n-1}$. We will assume throughout this work that all the  curvatures satisfy $\kappa_i(s)\not=0$ for any $s\in I$, $1\leq i\leq n-1$.

\begin{definition} A unit speed curve $\alpha:I\rightarrow\e^n$ is called cylindrical helix if its tangent vector $\v_1$ makes a constant angle with a fixed direction  $U$.
\end{definition}

Our main result in this work is the following characterization of cylindrical helices in Euclidean $n$-space $\e^n$.

\begin{theorem}\label{th-main} Let $\alpha:I\rightarrow\e^n$ be a unit speed curve in $\e^n$. Define the functions
\begin{equation}\label{u211}
\begin{array}{ll}
G_1=1,\,\,\,G_2=0,\,\,\,G_{i}=\dfrac{1}{\kappa_{i-1}}\Big[\kappa_{i-2}G_{i-2}+G_{i-1}^{\prime}\Big],\ 3\leq i\leq n.
\end{array}{}
\end{equation}
Then $\alpha$ is a cylindrical helix if and only if
the function
\begin{equation}\label{u21}
\sum_{i=3}^{n}\,G_i^2=C
\end{equation}
is constant. Moreover, the constant $C=\tan^2\theta$, being $\theta$ the angle that makes $\v_1$ with the fixed direction $U$ that determines $\alpha$.
\end{theorem}

This theorem generalizes in arbitrary dimensions what happens for $n=3$ and $n=4$, namely: if $n=3$, (\ref{u21}) writes
$G_3^2=\kappa_1/\kappa_2=\kappa/\tau$ and for $n=4$, (\ref{u21}) agrees with (\ref{eqma}).

\section{Proof of Theorem \ref{th-main}}

Let $\alpha$ be a unit speed curve in $\e^n$. Assume that $\alpha$ is a cylindrical helix curve. Let $U$ be the direction with which $\v_1$ makes a constant angle $\theta$ and, without loss of generality, we suppose that $\langle U,U\rangle=1$.  Consider the differentiable functions $a_i$, $1\leq i\leq n$,
\begin{equation}\label{u3}
U=\sum_{i=1}^{n}\,a_i(s)\,\v_i(s),\ \ s\in I,
\end{equation}
that is,
$$a_i=\langle\v_i,U\rangle,\ 1\leq i\leq n.$$
Then the function  $a_1(s)=\langle \v_1(s),U\rangle$ is constant, and it agrees with $\cos\theta$:
\begin{equation}\label{u311}
a_1(s)=\langle\v_1,U\rangle=\cos\theta
\end{equation}
for any $s$. By differentiation (\ref{u311}) with respect to $s$ and using the Frenet formula (\ref{u2}) we have
$$a_1'(s)=\kappa_1\,\langle\v_2,U\rangle=\kappa_1\,a_2=0.$$
Then $a_2=0$ and therefore $U$ lies in the subspace $Sp(\v_1,\v_3,\ldots,\v_n)$. Because the vector field $U$ is constant, a differentiation in (\ref{u3}) together (\ref{u2}) gives the following ordinary differential equation system
\begin{equation}\label{u5}
\left.\begin{array}{ll}
\kappa_1 a_1-\kappa_2 a_3&=0\\
a_3'-\kappa_3 a_4&=0\\
a_4'+\kappa_3 a_3-\kappa_4 a_5&=0\\
\vdots\\
a_{n-1}'+\kappa_{n-2} a_{n-2}-\kappa_{n-1} a_n&=0\\
a_n'+\kappa_{n-1} a_{n-1} &=0
\end{array}\right\}
\end{equation}
Define the functions $G_i=G_i(s=$ as follows
\begin{equation}\label{u51}
a_i(s)=G_i(s)\,a_1, \  3\leq i\leq n.
\end{equation}
We point out that $a_1\not=0$: on the contrary, (\ref{u5}) gives $a_i=0$, for $3\leq i\leq n$ and so, $U=0$: contradiction.
The first $(n-2)$-equations in (\ref{u5}) lead to
\begin{equation}\label{u6}
\left.
\begin{array}{ll}
&G_3=\dfrac{\kappa_1}{\kappa_2}\\
&G_4=\dfrac{1}{\kappa_3}G_3^{\prime}\\
&G_5=\dfrac{1}{\kappa_4}\Big[\kappa_3G_3+G_4^{\prime}\Big]\\
&\vdots\\
&G_{n-1}=\dfrac{1}{\kappa_{n-2}}\Big[\kappa_{n-3}G_{n-3}+G_{n-2}^{\prime}\Big]\\
&G_{n}=\dfrac{1}{\kappa_{n-1}}\Big[\kappa_{n-2}G_{n-2}+G_{n-1}^{\prime}\Big]
\end{array}\right\}
\end{equation}
The last equation of (\ref{u5}) leads to the following condition;
\begin{equation}\label{u62}
G_n'+\kappa_{n-1} G_{n-1}=0.
\end{equation}
We do the change of variables:
$$
t(s)=\int^s\kappa_{n-1}(u) du,\hspace*{.5cm}\frac{dt}{ds}=\kappa_{n-1}(s).
$$
In particular, and from the last equation of (\ref{u6}), we have
$$
G_{n-1}'(t)=G_n(t)-\Big(\frac{\kappa_{n-2}(t)}{\kappa_{n-1}(t)}\Big)G_{n-2}(t).
$$
As a consequence, if $\alpha$ is a cylindrical helix, substituting the equation (\ref{u62}) in the last equation yields
$$G_n''(t)+G_n(t)=\frac{\kappa_{n-2}(t)G_{n-2}(t)}{\kappa_{n-1}(t)}.$$
The general solution of this equation is
\begin{equation}\label{u9}
G_n(t)=\Big(A-\int\frac{\kappa_{n-2}(t)G_{n-2}(t)}{\kappa_{n-1}(t)}\sin{t}\,dt\Big)\cos{t}+
\Big(B+\int\frac{\kappa_{n-2}(t)G_{n-2}(t)}{\kappa_{n-1}(t)}\cos{t}\,dt\Big)\sin{t},
\end{equation}
where $A$ and $B$ are arbitrary constants. Then (\ref{u9}) takes the following form
\begin{equation}\label{u10}
\begin{array}{ll}
G_n(s)=&\Big(A-\int\Big[\kappa_{n-2}(s)G_{n-2}(s)\sin{\int\kappa_{n-1}(s)ds}\Big]ds\Big)\cos{\int\kappa_{n-1}(s)ds}\\
&+
\Big(B+\int\Big[\kappa_{n-2}(s)G_{n-2}(s)\cos{\int\kappa_{n-1}(s)ds}\Big]ds\Big)\sin{\int\kappa_{n-1}(s)ds}.
\end{array}
\end{equation}
From (\ref{u62}), the function $G_{n-1}$ is given by
\begin{equation}\label{u11}
\begin{array}{ll}
G_{n-1}(s)=&\Big(A-\int\Big[\kappa_{n-2}(s)G_{n-2}(s)\sin{\int\kappa_{n-1}(s)ds}\Big]ds\Big)\sin{\int\kappa_{n-1}(s)ds}\\
&-\Big(B+\int\Big[\kappa_{n-2}(s)G_{n-2}(s)\cos{\int\kappa_{n-1}(s)ds}\Big]ds\Big)\cos{\int\kappa_{n-1}(s)ds}.
\end{array}
\end{equation}
From Equation (\ref{u6}) and (\ref{u11}), we have
\begin{eqnarray*}
\sum_{i=3}^{n-2}G_i G_i^{\prime}&=&G_3\kappa_3G_4+G_4\Big(\kappa_4G_5-\kappa_3G_3\Big)+\ldots\\
&+& G_{n-3}\Big(\kappa_{n-3}G_{n-2}-\kappa_{n-4}G_{n-4}\Big)+G_{n-2}G_{n-2}^{\prime}\\
&=&G_{n-2}\Big(G_{n-2}'+\kappa_{n-3}G_{n-3}\Big)\\
&=&\kappa_{n-2}G_{n-2}G_{n-1}
\end{eqnarray*}
If we integrate the above equation, we have
\begin{equation}\label{u15}
\begin{array}{ll}
\sum_{i=3}^{n-2}G_{i}^{2}&=C-\Big(A-\int\Big[\kappa_{n-2}(s)G_{n-2}(s)\sin{\int\kappa_{n-1}ds}\Big]ds\Big)^2\\
&-
\Big(B+\int\Big[\kappa_{n-2}(s)G_{n-2}(s)\cos{\int\kappa_{n-1}ds}\Big]ds\Big)^2,
\end{array}
\end{equation}
where $C$ is a constant of integration. From Equations (\ref{u10}) and (\ref{u11}), we have
\begin{equation}\label{u151}
\begin{array}{ll}
G_n^2+G_{n-1}^2&=\Big(A-\int\Big[\kappa_{n-2}(s)G_{n-2}(s)\sin{\int\kappa_{n-1}ds}\Big]ds\Big)^2\\
&+
\Big(B+\int\Big[\kappa_{n-2}(s)G_{n-2}(s)\cos{\int\kappa_{n-1}ds}\Big]ds\Big)^2,
\end{array}
\end{equation}
It follows from  (\ref{u15}) and (\ref{u151}) that
$$\sum_{i=3}^{n}G_{i}^{2}=C.$$
Moreover this constant $C$ calculates as follows.
From (\ref{u51}), together the $(n-2)$-equations (\ref{u6}), we have
$$C=\sum_{i=3}^{n}\,G_i^2=\dfrac{1}{a_1^2}\sum_{i=3}^{n}\,a_i^2=\dfrac{1-a_1^2}{a_1^2}=\tan^2\theta,$$
where we have used (\ref{u2}) and the fact that $U$ is a unit vector field.

We do the converse of Theorem. Assume that the condition (\ref{u6}) is satisfied for a curve $\alpha$. Let $\theta\in\r$ be so that $C=\tan^2\theta$. Define the unit vector $U$  by
$$U=\cos\theta\Big[\v_1+\sum_{i=3}^{n}\,G_i\,\v_i\Big].$$
By taking account (\ref{u6}), a differentiation of $U$ gives that $\dfrac{dU}{ds}=0$, which it means that $U$ is a constant vector field. On the other hand, the scalar product between the unit tangent vector field $\v_1$ with $U$ is
$$
\langle\v_1(s),U\rangle=\cos\theta.
$$
Thus $\alpha$ is a cylindrical helix curve. This finishes  the proof of Theorem \ref{th-main}.

As a direct consequence of the proof, we generalize Theorem \ref{th-main} in Minkowski space and  for timelike curves.

\begin{theorem} Let $\e_1^n$ be the Minkowski n-dimensional space  and let $\alpha:I\rightarrow\e_1^n$ be a unit speed timelike curve.
Then $\alpha$ is a cylindrical helix if and only if
the function $\sum_{i=3}^{n}\,G_i^2$ is constant, where the functions $G_i$ are defined as in (\ref{u211}).
\end{theorem}
\begin{proof} The proof carries the same steps as above and we omit the details. We only point out that the fact that $\alpha$ is timelike means that $\v_1(s)=\alpha'(s)$ is a timelike vector field. The other $V_i$ in the Frenet frame, $2\leq i\leq n$, are unit spacelike vectors and so, the second equation in (\ref{u2}) changes to $\v_2'=\kappa_1\v_1+\kappa_2\v_3$ (\cite{al1,ps}).

\end{proof}

\section{Further characterizations of cylindrical helices}

In this section we present  new characterizations of cylindrical helix in $\e^n$. The first one is a consequence of Theorem \ref{th-main}.

\begin{theorem} \label{th-2} Let $\alpha:I\subset R\rightarrow \e^n$ be a unit speed curve in Euclidean space $\e^n$. Then $\alpha$ is a cylindrical helix if and only if there exists a $C^2$-function $G_{n}(s)$ such that
\begin{equation}\label{u24}
G_{n}=\dfrac{1}{\kappa_{n-1}}\Big[\kappa_{n-2}G_{n-2}+G_{n-1}^{\prime}\Big],\,\,\,
\dfrac{dG_{n}}{ds}=-\kappa_{n-1}(s)G_{n-1}(s),
\end{equation}
where
$$G_1=1, G_2=0,G_{i}=\dfrac{1}{\kappa_{i-1}}\Big[\kappa_{i-2}G_{i-2}+G_{i-1}^{\prime}\Big],\ 3\leq i\leq n-1.$$
\end{theorem}

\begin{proof}
Let now assume that $\alpha$ is a cylindrical helix. By using Theorem \ref{th-main} and by differentiation the (constant) function  given in
(\ref{u21}), we obtain

\begin{eqnarray*}
0&=&\sum_{i=3}^{n}G_i\,G_i^{\prime}\\
&=&G_3\kappa_3G_4+G_4\Big(\kappa_4G_5-\kappa_3G_3\Big)+...+G_{n-1}\Big(\kappa_{n-1}G_{n}-\kappa_{n-2}G_{n-2}\Big)+G_{n}G_{n}^{\prime}\\
&=&G_{n}\Big(G_{n}'+\kappa_{n-1}G_{n-1}\Big).
\end{eqnarray*}
This shows (\ref{u24}). Conversely, if (\ref{u24}) holds, we define a  vector field $U$ by
$$U=\cos\theta\Big[\v_1+\sum_{i=3}^{n}\,G_i\,\v_i\Big].$$
By the Frenet equations (\ref{u2}), $\frac{dU}{ds}=0$, and so, $U$ is constant. On the other hand,  $\langle\v_1(s),U\rangle=\cos\theta$ is constant, and this means that $\alpha$ is a cylindrical helix.
\end{proof}

We end  giving an integral characterization of a cylindrical helix.

\begin{theorem} Let  $\alpha:I\subset R\rightarrow \e^n$ be a unit speed curve in Euclidean space $\e^n$. Then $\alpha$ is a cylindrical helix if and only if the following condition is satisfied
\begin{equation}\label{u244}
\begin{array}{ll}
G_{n-1}(s)=&\Big(A-\int\Big[\kappa_{n-2}G_{n-2}\sin{\int\kappa_{n-1}ds}\Big]ds\Big)\sin{\int^s\kappa_{n-1}(u)du}\\
&-\Big(B+\int\Big[\kappa_{n-2}G_{n-2}\cos{\int\kappa_{n-1}ds}\Big]ds\Big)\cos{\int^s\kappa_{n-1}(u)du}.
\end{array}
\end{equation}
for some constants $A$ and $B$.
\end{theorem}
\begin{proof}
Suppose that $\alpha$ is a cylindrical helix. By using  Theorem \ref{th-2}, let define $m(s)$ and $n(s)$ by
$$\phi(s)=\int^s\kappa_{n-1}(u)du,$$
\begin{equation}\label{u26}
\begin{array}{ll}
m(s)=&G_{n}(s)\cos\phi+G_{n-1}(s)\sin\phi+\int\,\kappa_{n-2}G_{n-2}\sin\phi\,ds,\\
n(s)=&G_{n}(s)\sin\phi-G_{n-1}(s)\cos\phi-\int\,\kappa_{n-2}G_{n-2}\cos\phi\,ds.
\end{array}
\end{equation}
If we differentiate equations (\ref{u26}) with respect to $s$ and taking into  account of (\ref{u244}) and (\ref{u24}), we
obtain  $\dfrac{dm}{ds}=0$ and $\dfrac{dn}{ds}=0$. Therefore, there exist constants $A$ and $B$ such that $m(s)=A$ and $n(s)=B$.
By substituting  into (\ref{u26}) and solving the resulting equations for $G_{n-1}(s)$, we get
$$
G_{n-1}(s)=\Big(A-\int\,\kappa_{n-2}G_{n-2}\sin{\phi}\,ds\Big)\sin{\phi}-
\Big(B+\int\,\kappa_{n-2}G_{n-2}\cos{\phi}\,ds\Big)\cos{\phi}.
$$

Conversely, suppose that (\ref{u244}) holds. In order to apply Theorem \ref{th-2}, we define $G_n(s)$ by
$$G_n(s)=\Big(A-\int\,\kappa_{n-2}G_{n-2}\sin{\phi}\,ds\Big)\cos{\phi}+
\Big(B+\int\,\kappa_{n-2}G_{n-2}\cos{\phi}\,ds\Big)\sin{\phi}.$$
with $\phi(s)=\int^s\kappa_{n-1}(u) du$. A direct differentiation of (\ref{u244}) gives
$$
G_{n-1}^{\prime}=\kappa_{n-1}G_{n}-\kappa_{n-2}G_{n-2}.
$$
This shows the left condition in (\ref{u24}). Moreover, a straightforward computation leads to
$G_{n}'(s)=-\kappa_{n-1}G_{n-1}$, which finishes the proof.
\end{proof}

We end this section with a characterization of cylindrical helices only in terms of the curvatures of $\alpha$. From the definitions of $G_i$ in (\ref{u211}),  one can express the functions $G_i$ in terms of $G_3$ and the curvatures of $\alpha$ as follows:
\begin{equation}\label{u261}
\begin{array}{ll}
G_j=\sum_{i=0}^{j-3}\,A_{ji}G_3^{(i)},\,\,\,3\leq j \leq n,
\end{array}
\end{equation}
where
$$
G_3^{(i)}=\dfrac{d^{(i)}G_3}{ds^i},\,\,\,G_3^{(0)}=G_3=\dfrac{\kappa_1}{\kappa_2}.
$$
Then
$$\begin{array}{ll}
G_4=\kappa_3^{-1}G_3'=A_{41}G_3'+A_{40}G_3,&  A_{41}=\kappa_3^{-1}, A_{40}=0\\
G_5=A_{52}G_3''+A_{51}G_3'+A_{50}G_3, & A_{52}=\kappa_{4}^{-1}A_{41}, A_{51}=\kappa_{4}^{-1}A_{41}', A_{50}=\kappa_{4}^{-1}\kappa_3
\end{array}$$
and so on.  Define the following functions:
$$A_{30}=1, A_{40}=0$$
$$ A_{j0}=\kappa_{j-1}^{-1}\kappa_{j-2}A_{(j-2)0}+\kappa_{j-1}^{-1}A_{(j-1)0}',\ 5\leq j \leq n$$
$$\begin{array}{lll}
A_{j(j-3)}&=\kappa_{j-1}^{-1}\kappa_{j-2}^{-1}\kappa_{j-3}^{-1}...\kappa_{4}^{-1}\kappa_{3}^{-1},& 4\leq j \leq n\\
A_{j(j-4)}&=\kappa_{j-1}^{-1}\Big(\kappa_{j-2}^{-1}\kappa_{j-3}^{-1}\ldots\kappa_{4}^{-1}\kappa_{3}^{-1}\Big)'+
\kappa_{j-1}^{-1}\kappa_{j-2}^{-1}\Big(\kappa_{j-3}^{-1}\ldots\kappa_{4}^{-1}\kappa_{3}^{-1}\Big)' & \\
&+\ldots+\kappa_{j-1}^{-1}\kappa_{j-2}^{-1}\kappa_{j-3}^{-1}\ldots\kappa_{4}^{-1}\Big(\kappa_{3}^{-1}\Big)', &5\leq j \leq n\\
A_{ji}&=\kappa_{j-1}^{-1}\kappa_{j-2}A_{(j-2)i}+\kappa_{j-1}^{-1}\Big(
A_{(j-1)i}'+A_{(j-1)(i-1)}
\Big)& 1\leq  i \leq j-5,6\leq j \leq n
\end{array}$$
and $A_{ji}=0$ otherwise.

The second equation of (\ref{u24}), leads the following condition:
\begin{equation}\label{u278}
\begin{array}{ll}
A_{n(n-3)}G_{3}^{(n-2)}&+\Big(A_{n(n-3)}'+A_{n(n-4)}\Big)G_{3}^{(n-3)}\\
&+\sum_{i=1}^{n-4}\Big[A_{ni}'+A_{n(i-1)}+\kappa_{n-1}A_{(n-1)i}\Big]G_3^{(i)}\\
&+\Big(A_{n0}'+\kappa_{n-1}A_{(n-1)0}\Big)G_3=0,\,\,\,n\geq 3.
\end{array}
\end{equation}

As a consequence of (\ref{u278}) and Theorem \ref{th-main}, we have  the following

\begin{corollary} Let $\alpha:I\rightarrow\e^n$ be a unit speed curve in $\e^n$. The next statements are equivalent:
\begin{enumerate}
\item  $\alpha$ is a cylindrical helix.
\item $$\begin{array}{ll}
0&=A_{n(n-3)}\Big(\dfrac{\kappa_1}{\kappa_2}\Big)^{(n-2)}+\Big(A_{n(n-3)}'+A_{n(n-4)}\Big)
\Big(\dfrac{\kappa_1}{\kappa_2}\Big)^{(n-3)}\\
&+\sum_{i=1}^{n-4}\Big[A_{ni}'+A_{n(i-1)}+\kappa_{n-1}A_{(n-1)i}\Big]\Big(\dfrac{\kappa_1}{\kappa_2}\Big)^{(i)}\\
&+\Big(A_{n0}'+\kappa_{n-1}A_{(n-1)0}\Big)\Big(\dfrac{\kappa_1}{\kappa_2}\Big),\hspace*{.5cm}n\geq 3.
\end{array}$$
\item The function
$$\begin{array}{ll}
\sum_{j=3}^{n}\,\sum_{i=0}^{j-3}\,\sum_{k=0}^{j-3}\,A_{ji}A_{jk}\Big(\dfrac{\kappa_1}{\kappa_2}\Big)^{(i)}
\Big(\dfrac{\kappa_1}{\kappa_2}\Big)^{(k)}=C
\end{array}$$
is constant, $j-i\geq 3$, $j-k\geq 3$.
\end{enumerate}
\end{corollary}


\end{document}